\documentclass{amsart}
\usepackage{amsmath}
\usepackage{amssymb}
\usepackage{amsthm}
\usepackage{dsfont}
\usepackage{lineno}
\usepackage{subfigure}
\usepackage{graphicx}
\usepackage{multirow}
\usepackage{color}
\usepackage{xspace}
\usepackage{pdfsync}
\usepackage{hyperref} 
\usepackage{algorithmicx}
\usepackage{algpseudocode}
\usepackage{algorithm}
\usepackage{mathtools}
\usepackage{enumitem}
\graphicspath{{Figures/}}

\DeclareMathOperator*{\argmin}{argmin}
\DeclareMathOperator*{\argmax}{argmax}
\DeclarePairedDelimiter\ceil{\lceil}{\rceil}

\newcommand{\bq}{\begin{equation}}
\newcommand{\eq}{\end{equation}}
\newcommand{\R}{\mathbb{R}}
\newcommand{\Z}{\mathbb{Z}}
\newcommand{\abs}[1]{\left\vert#1\right\vert}

\newcommand{\G}{\mathcal{G}}
\newcommand{\bO}{\mathcal{O}}
\newcommand{\dist}{\text{dist}}

\newcommand{\Dt}{\mathcal{D}}

\newcommand{\Af}{\mathcal{A}}
\newcommand{\Sf}{\mathcal{S}}
\newcommand{\Mf}{\mathcal{M}}
\newcommand{\MA}{Monge-Amp\`ere\xspace}

\algnewcommand{\LineComment}[1]{\State \(\triangleright\) #1}

\newtheorem{theorem}{Theorem}

\theoremstyle{lemma}
\newtheorem{lemma}[theorem]{Lemma}

\newtheorem{corollary}[theorem]{Corollary}

\newtheorem{definition}[theorem]{Definition}

\newtheorem{remark}[theorem]{Remark}

\newtheorem{hypothesis}[theorem]{Hypothesis}

\theoremstyle{remark}

\makeatletter
\newcommand\appendix@section[1]{%
\refstepcounter{section}%
\orig@section*{Appendix \@Alph\c@section: #1}%
}
\let\orig@section\section
\g@addto@macro\appendix{\let\section\appendix@section}
\makeatother

\begin{document}

\title[Second boundary value problem for Monge-Amp\`ere equation]{Convergence framework for the second boundary value problem for the Monge-Amp\`ere equation}

\author{Brittany Froese Hamfeldt}
\address{Department of Mathematical Sciences, New Jersey Institute of Technology, University Heights, Newark, NJ 07102}
\email{bdfroese@njit.edu}
\thanks{This work was partially supported by NSF DMS-1619807 and NSF DMS-1751996.}

\begin{abstract}
  It is well known that the quadratic-cost optimal transportation problem is formally equivalent to the second boundary value problem for the Monge-Amp\`ere equation.  Viscosity solutions are a powerful tool for analysing and approximating fully nonlinear elliptic equations.  However, we demonstrate that this nonlinear elliptic equation does not satisfy a comparison principle and thus existing convergence frameworks for viscosity solutions are not valid.  We introduce an alternative PDE that couples the usual Monge-Amp\`ere equation to a Hamilton-Jacobi equation that restricts the transportation of mass.  We propose a new interpretation of the optimal transport problem in terms of viscosity subsolutions of this PDE.  Using this reformulation, we develop  a framework for proving convergence of a large class of approximation schemes for the optimal transport problem.  Examples of existing schemes that fit within this framework are discussed.
\end{abstract}

\date{\today}    
\maketitle

The goal of optimal transportation is to find a transport plan $T(x)$ that rearranges a distribution $f$ into a second distribution $g$, while minimising some transport cost.  In the most commonly studied case, the cost is quadratic in the displacement and the goal is to compute the optimal mapping
\bq\label{eq:OT}
T(x) = \argmin\limits_{T \in \Mf}\int_X \abs{x-T(x)}^2 f(x)\,dx 
\eq
where
\[ \Mf = \left\{T:X\to Y \mid \int_E f(x)\,dx = \int_{T(E)} g(y)\,dy \text{ for all measurable } E \subset X \right\}. \]

Formally, the optimal mapping can be characterized as the gradient of a convex function $u$, which is given as the solution of a \MA equation
\[ g(\nabla u(x))\det(D^2u(x)) = f(x) \]
equipped with the constraint
\[ \nabla u(X) \subset \bar{Y}. \]
This is known as the \emph{second boundary value problem for the \MA equation}.
This PDE has been studied under strong hypotheses that yield classical solutions~\cite{Caf_BVP2, Delanoe, Urbas_BVP2} and much weaker hypotheses that yield generalised solutions~\cite{Villani}.  However, existing solution notions do not lend themselves naturally to numerical approximation.

In the case of semi-discrete optimal transportation (in which the density $f$ is replaced with a collection of Dirac masses), a generalised interpretation of the \MA equation led to the earliest numerical methods for optimal transport~\cite{Aurenhammer_minkowski,olikerprussner88}, and has more recently been developed into a very efficient method~\cite{Levy_OT}. This approach can  be applied to the continuous problem through quantisation of the continuous density $f$.  A proof of convergence for the semi-discrete problem, including convergence rates, is also available~\cite{Berman_convergence}.  

In this work, we are concerned with optimal transport involving continuous measures (densities).  A powerful tool for studying fully nonlinear second order elliptic equations is the viscosity solution~\cite{CIL}, which uses a maximum principle argument to transfer derivatives onto smooth test functions.  This concept of weak solution is particularly useful for suggesting appropriate numerical methods~\cite{BSnum,FroeseMeshfreeEigs,FOFiltered,ObermanEigenvalues}.  Several numerical methods have been proposed for the second boundary value problem for the \MA equation~\cite{BFO_OTNum,FroeseTransport,Prins_BVP2}.  However, these methods lack any proof of convergence.

Recently, two new results have appeared related to the convergence of numerical methods for the second boundary value problem for the Monge-Amp\`ere equation~\cite{BenamouDuval_MABVP2,LindseyRubinstein}.  Both convergence results hinge on two important observations: 1) the \MA equation enforces volume conservation, and it is sufficient for the discrete version of this operator to overestimate the true value (interpreted in an appropriate weak sense) and 2) the stability of optimal transportation ensures that the solutions to a sequence of ``nearby'' optimal transport problems converge to the desired solution.  

The first of these methods describes an optimisation problem that enforces over-estimation of the \MA operator via the objective function, with convexity of the solution and the second boundary condition included as constraints~\cite{LindseyRubinstein}.  The optimisation framework leads to a very robust approach that allows for proof that the discrete problem is well-posed and that solutions converge to the true weak solution of the optimal transport problem.  A downside to this framework is that the convexity of the solution is enforced as a global constraint, which is expensive to implement numerically. 

The second of these methods relies on a measure theory interpretation of the \MA equation to produce a discretisation of the \MA equation and second boundary condition~\cite{BenamouDuval_MABVP2}.  This leads to a fully local, computationally efficient numerical method.  Because of the natural interpretation of this approximation in terms of measures, the authors successfully prove that solutions of the discrete problem will converge to the solutions to the true problem.  However, existence of solutions to the discrete problem is left as an open question.

The goal of the present article is to develop a general framework for proving the convergence of approximation schemes for the quadratic-cost optimal transport problem via the second boundary value problem for the Monge-Amp\`ere equation.  The usual general techniques for proving uniqueness and convergence rely on a comparison principle that is demonstrably false for this equation.  We introduce an alternative form of the PDE that, while it does not satisfy a comparison principle, does place strong constraints on subsolutions.  As in~\cite{BenamouDuval_MABVP2,LindseyRubinstein}, we make the key observations that the \MA operator need only be enforced as a constraint and that optimal transportation is stable.  Using these properties and the more general theory of viscosity solutions, we show that subsolutions are equivalent to generalised solutions of the optimal transportation problem.  We then describe a general framework for analysing the convergence of approximation schemes.  In particular, schemes that are consistent, monotone, and under-estimating are guaranteed to be well-posed and to converge.  Several existing schemes for the \MA equation such as~\cite{benamou2014monotone,BenamouDuval_MABVP2,BFO_OTNum,FinlayOberman,FroeseMeshfreeEigs,FO_MATheory,HS_Quadtree,Nochetto_MAConverge}, after slight modification, fit within this framework and are thus guaranteed to compute the weak solution of the original optimal transportation problem.

\section{Background}\label{sec:background}

\subsection{Second boundary value problem}\label{sec:bvp2}
Due to key results by Brenier~\cite{Brenier_polar} and Rockafellar~\cite{Rockafellar_convex}, the optimal transport plan can be characterised as the (sub)gradient of a convex function, $T(x) = \nabla u(x)$.  Combining this with conservation of mass leads to an elliptic \MA equation
\bq\label{eq:MA}
g(\nabla u(x)) \det(D^2u(x)) = f(x), \quad u \text{ is convex.} 
\eq
Instead of being coupled to a traditional boundary condition, this equation is augmented with a global constraint on the solution (sub)gradient,
\bq\label{eq:bvp2}
\partial u(X) \subset \bar{Y},
\eq
which leads to the so-called \emph{second boundary value problem} for the \MA equation.

By introducing a \emph{defining function} for the target set $Y$, it is possible to re-express~\eqref{eq:bvp2} as a formally equivalent nonlinear Neumann boundary condition~\cite{Delanoe}.
\begin{definition}[Defining function]\label{def:defining}
A \emph{defining function} for the set $Y \subset \R^n$ is a continuous function $H(y)$ satisfying
\[
H(y) \begin{cases}
<0 & y \in Y \\
=0 & y \in \partial Y\\
>0 & y \notin \bar{Y}.
\end{cases}
\]
\end{definition}
A natural choice of defining function is the signed distance function.

This is used to rewrite the global constraint~\eqref{eq:bvp2} as
\bq\label{eq:HJBC}
H(\nabla u(x)) = 0, \quad x \in \partial X,
\eq
which requires that boundary points be mapped to boundary points.  When the data is sufficiently smooth, with density functions supported in uniformly convex sets $X, Y$ and bounded away from 0 and $\infty$, the second boundary value problem~\eqref{eq:MA}, \eqref{eq:HJBC} admits a smooth classical solution~\cite{Caf_BVP2, Urbas_BVP2}.  If the equation is augmented by an additional condition such as a mean-zero condition $\langle u \rangle = 0$, the solution is unique.

In general, smooth solutions do not exist and some notion of weak solution is needed in order to properly interpret solutions of the optimal transport problem using the \MA equation.  A powerful option is the Aleksandrov solution, which relies on the fact that the subgradient of the convex potential $u$ can be used to define a measure.

\begin{definition}[Aleksandrov solution]\label{def:aleks}
We say that a convex function $u$ is an \emph{Aleksandrov solution} of the \MA equation~\eqref{eq:MA} if
\bq\label{eq:aleks}
\int_E f(x)\,dx = \int_{\partial u(E)} g(y)\,dy
\eq
for every measurable set $E \subset X$.
\end{definition}

We will consider the problem under the following hypotheses on the data.

\begin{hypothesis}[Conditions on data]\label{hyp}
\end{hypothesis}

\begin{enumerate}
\item[(H1)] $X, Y$ are convex, bounded, open domains.
\item[(H2)] The source density $f\in L^1(X)$ is non-negative and lower semicontinuous.
\item[(H3)] The target density $g\in L^1(\R^n)$ is positive on $Y$, vanishes on $Y^c$, and upper semicontinuous.
\item[(H4)] The data satisfies the mass balance condition
\bq\label{eq:massbalance}  \int_X f(x)\,dx = \int_Y g(y)\,dy.\eq
\end{enumerate}

We remark that while the density function $g$ must be positive, it does not need to be bounded away from zero.  The source density $f$, on the other hand, is allowed to vanish, and it can be supported on a non-convex set.

Under these conditions, the second boundary value problem for the \MA equation, interpreted in the Aleksandrov sense, is equivalent to the solution of the optimal transportation problem.


\begin{theorem}[Existence of Aleksandrov solution~{\cite[Theorems 2.12 and 4.10]{Villani}}]\label{thm:aleksExist}
Under Hypothesis~\ref{hyp}, there exists an Aleksandrov solution $u$ of the \MA equation~\eqref{eq:MA} that satisfies the second boundary constraint~\eqref{eq:bvp2}.  Moreover, this solution is uniquely defined on $\text{supp}(f)$ up to additive constants and the subgradient map $\partial u$ solves the optimal transport problem~\eqref{eq:OT}.
\end{theorem}

The existence of Aleksandrov solutions to the Dirichlet problem is slightly more delicate, but will play an important role in building up a robust understanding of viscosity solutions of the second boundary value problem.  The following key result is a special case of a more general theorem due to Bakelmen.
\begin{theorem}[Aleksandrov solutions of the Dirichlet problem~{\cite[Theorem 12.1]{Bakelman_Elliptic}}]\label{thm:aleksDirichlet}
Let $X$ be a uniformly convex, bounded, domain.  Suppose that $f$ is continuous and non-negative on $\bar{X}$  with 
\[ f(x) \leq C_1\text{dist}(x,\partial X), \quad x \in X\cap U \]
where $C_1>0$ is a constant and $U$ is some neighbourhood of the boundary $\partial X$.  Suppose also that $g(y)$ is continuous and satisfies
\[ g(y) \geq C_2, \quad y \in \R^n \]
for some constant $C_2>0$.  Let $h$ be a continuous function on $\partial X$. Then there exists at least one continuous Aleksandrov solution $u(x)$ of the \MA equation~\eqref{eq:MA} such that $u(x) = h(x)$ on $\partial X$.
\end{theorem}

%

\subsection{Viscosity solutions}\label{sec:viscosity}

The \MA equation belongs to a class of PDEs known as degenerate elliptic equations, which take the form
\[ F(x,u(x),\nabla u(x),D^2u(x)) = 0. \]

\begin{definition}[Degenerate elliptic]\label{def:elliptic}
The operator
$F:X\times\R\times\R^n\times\Sf^n\to\R$
is \emph{degenerate elliptic} if 
\[ F(x,u,p,A) \leq F(x,v,p,B) \]
whenever $u \leq v$ and $A \geq B$.
\end{definition}

The notion of the viscosity solution has become a very powerful tool for analysing fully nonlinear degenerate elliptic PDEs~\cite{CIL}.  The definition relies on a maximum principle argument that moves derivatives onto smooth test functions.  

Viscosity solutions of the equation 
\bq\label{eq:PDE} F(x,u(x),\nabla u(x),D^2 u(x))  = 0 \eq
are defined as follows.
\begin{definition}[Viscosity solution]\label{def:viscosity}
An upper (lower) semi-continuous function $u$ is a \emph{viscosity sub(super)solution} of~\eqref{eq:PDE} in $X$ if for every $\phi\in C^2$, whenever $u-\phi$ has a local maximum (minimum) at $x \in X$, then
\[ 
F_*^{(*)}(x,u(x),\nabla \phi(x),D^2\phi(x)) \leq (\geq)  0 
\]
where $F_*^{(*)}$ denotes the lower (upper) semi-continuous envelope of $F$.

A continuous function $u$ is a \emph{viscosity solution} of~\eqref{eq:PDE} if it is both a viscosity subsolution and a viscosity supersolution.
\end{definition}

By extending the operator $F$ to the boundary of the domain, we can also interpret boundary conditions in the viscosity sense.

It is not hard to show that Aleksandrov solutions of the \MA equation are also viscosity solutions of the \MA equation in open sets.  We obtain the following result through a trivial adaptation of~\cite[Proposition 1.3.4]{Gutierrez}
\begin{theorem}[Aleskandrov solutions are viscosity solutions]\label{thm:aleksVisc}
Let $u$ be a convex Aleksandrov solution of the \MA equation in $X$.  Then $u$ is a viscosity solution of the \MA equation in $X$.
\end{theorem}
The converse of this theorem cannot be so trivially adapted from the results of~\cite{Gutierrez}, particularly since we allow density functions that vanish and/or are discontinuous.

\subsection{Comparison principle}\label{sec:comparison}
A key property of many elliptic operators is a comparison principle, which is used to prove uniqueness and existence results, and plays a key role in proving that monotone approximation schemes are convergent~\cite{BSnum}.  The classical form of the comparison principle allows us to compare subsolutions and supersolutions in the interior of the domain using information from the boundary.
Under appropriate assumptions on the data, the \MA equation does indeed possess this form of comparison principle.
\begin{theorem}[Comparison principle for \MA~{\cite[Theorem V.2]{IshiiLions}}]\label{thm:comparisonMA}
Let $X$ be a convex domain and let $f, g$ be non-negative and locally Lipschitz continuous with $g$ bounded away from zero.  Suppose that $u$ is a viscosity subsolution and $v$ a viscosity supersolution of the \MA equation. Then
\[ \sup\limits_X\{u-v\} = \sup\limits_{\partial X}\{u-v\}. \]
\end{theorem}

However, the powerful Barles-Souganidis convergence framework~\cite{BSnum}--and indeed many uniqueness results--require a stronger form of the comparison principle that involves interpreting the boundary conditions in the viscosity sense.
\begin{definition}[Comparison principle]\label{def:comparisonStrong}
Let the PDE operator $F$ be defined on $\bar{X}\times \R \times \R^n \times \Sf^n$.  The PDE~\eqref{eq:PDE} has a comparison principle if whenever $u$ is a viscosity subsolution and $v$ a viscosity supersolution then $u \leq v$ on $\bar{X}$.
\end{definition}

Unfortunately, many elliptic PDEs do not satisfy this strong form of the comparison principle.  For the Dirichlet problem, for example, the \MA equation does not always admit a continuous solution.  In this setting, the Dirichlet boundary condition is interpreted in a weak sense, and subsolutions need not lie below super-solutions.  Here, the violation of the comparison principle is relatively mild, occurring only at boundary points.  In the interior, sub and super-solutions remain ordered, and monotone approximation schemes are guaranteed to converge in the interior of the domain~\cite{Hamfeldt_Gauss}.

The situation becomes much more delicate for the second boundary value problem.  To illustrate, we consider the one-dimensional problem of mapping the uniform density $f(x)=1$ on the line segment $X = (-1,1)$ back onto itself ($g=f$ and $Y=X$).  We use the convex defining function $H(y) = \abs{y}-1$.  The solution to this problem is the identity map, which has potential
\bq\label{eq:sol1d} u(x) = \frac{1}{2}x^2 + C \eq
where $C$ is any constant.

In order to obtain a unique solution (which is certainly necessary for the comparison principle), we need to include an additional condition.  Many options are possible; here we consider three popular options.

\bq\label{eq:ex1}\tag{Ex. 1}
F(x,u,u',u'') = 
\begin{cases}
-u'' + 1, & x \in (-1,1)\\
\abs{u'}-1-\langle u \rangle, & x = \pm 1.
\end{cases}
\eq

\bq\label{eq:ex2}\tag{Ex. 2}
F(x,u,u',u'') = 
\begin{cases}
-u'' + 1, & x \in (-1,1)\\
\abs{u'}-1-u(1), & x = \pm 1.
\end{cases}
\eq

\bq\label{eq:ex3}\tag{Ex. 3}
F(x,u,u',u'') = 
\begin{cases}
-u'' + 1, & x \in (-1,1)\\
\abs{u'}-1, & x = \pm 1
\end{cases}
\quad\quad \langle u \rangle = 0.
\eq

In the first example, the boundary points are mapped onto points satisfying $H(y) = \langle u \rangle$, and the mean-zero condition $\langle u\rangle = 0$ is enforced indirectly via mass balance. Similarly, example two fixes the value of one point $u(1) = 0$.  In the third example, the operator is defined only on functions satisfying the mean-zero condition.  In each case, the quadratic function~\eqref{eq:sol1d} satisfies the boundary value problem with $C = -1/6$ (examples 1 and 3) or $C = -1/2$ (example 2).  

We remark that all of these operators actually include a non-local part (either the average value of $u$ or $u$ at a fixed location in the domain).  Thus none of these strictly satisfy the usual definition of a (local) elliptic operator (Definition~\ref{def:elliptic}).  In particular, the modification necessary to ensure uniqueness sacrifices the comparison principle.

We notice that the functions $u(x)$ are solutions, and therefore also subsolutions, of the second boundary value problem.  We can also construct super-solutions.  For example, the function $v(x) = -2x$ is a super-solution of all three operators (and $\langle v \rangle = 0$, so it is admissible in example 3).  However, it is \emph{not} the case that $u(x) \leq v(x)$ in $(-1,1)$; see Figure~\ref{fig:noComparison}.  That is, these characterisations of the second boundary value problem do \emph{not} satisfy a comparison principle.

\begin{figure}
{\subfigure[]{\includegraphics[width=0.45\textwidth]{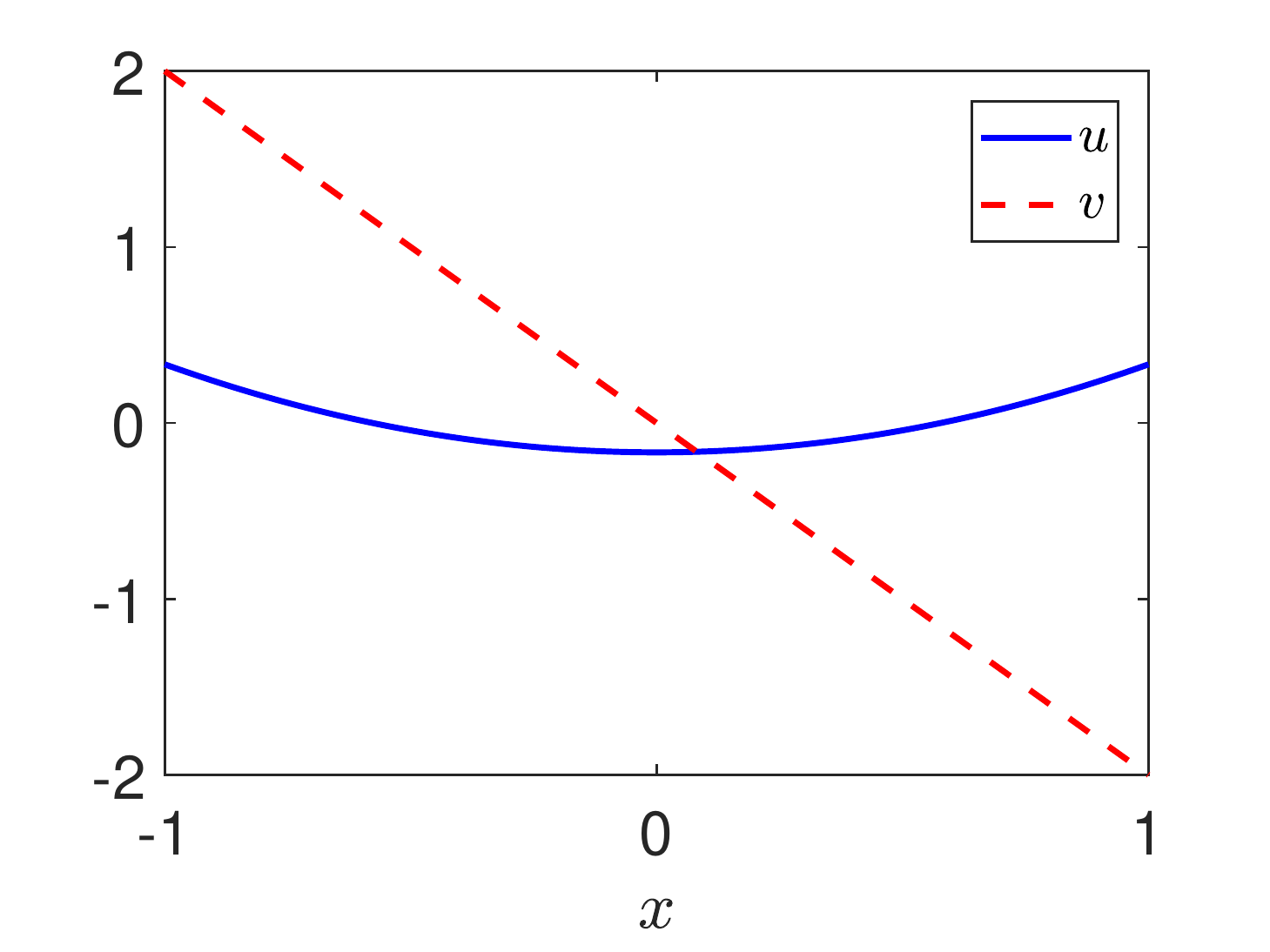}\label{fig:ex1}}}
{\subfigure[]{\includegraphics[width=0.45\textwidth]{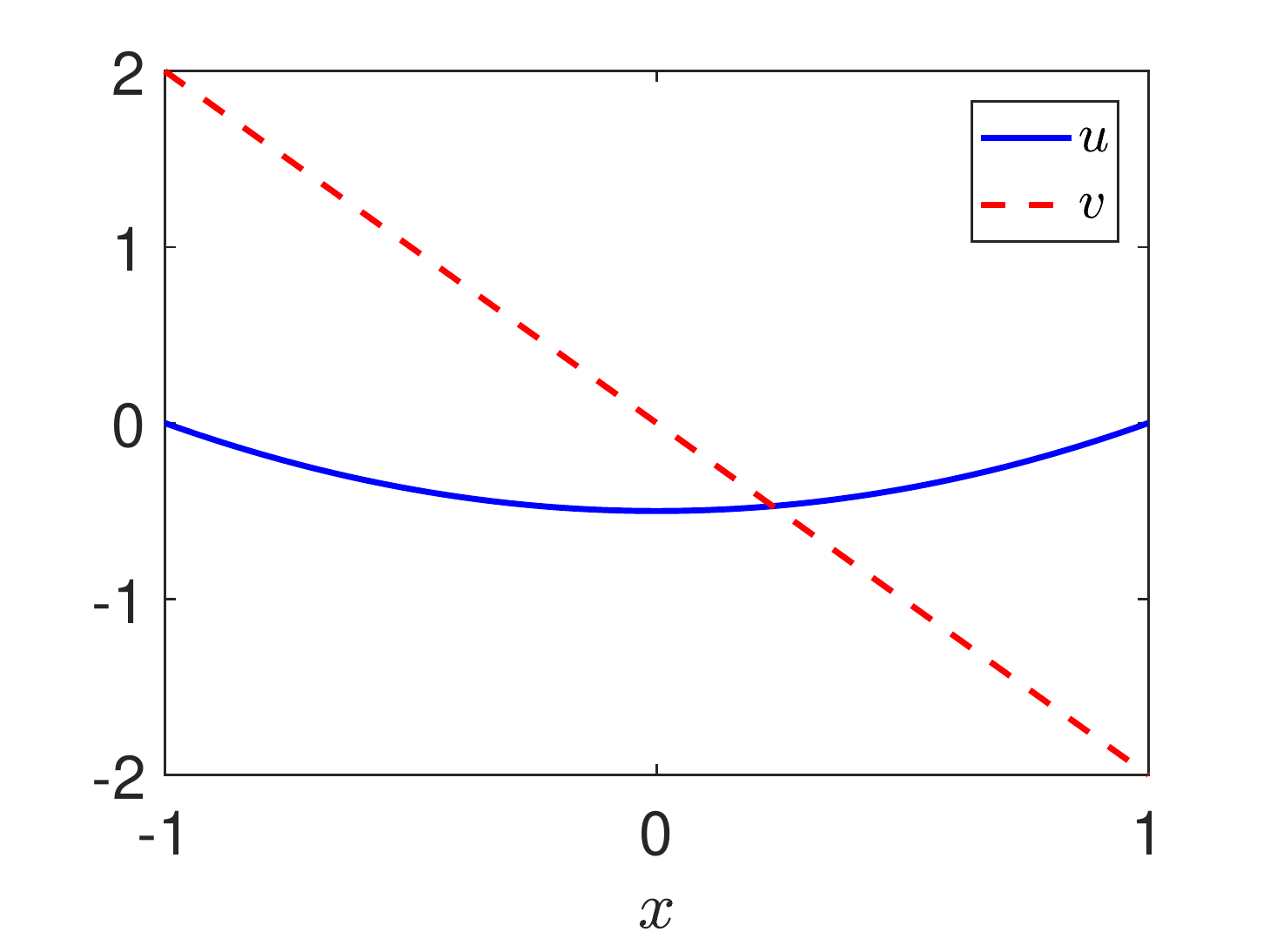}\label{fig:ex2}}}
\caption{Subsolution $u$ and supersolution $v$ of \subref{fig:ex1}~\eqref{eq:ex1}, \eqref{eq:ex3}, and \subref{fig:ex2}~\eqref{eq:ex2}.}
\label{fig:noComparison}
\end{figure}

\section{Alternate form of the PDE}\label{sec:subsolutions}
We propose an alternate form of the PDE that incorporates the Monge-Amp\`ere equation and boundary constraint into a single equation posed throughout the domain. As in earlier examples, this new equation does not satisfy a comparison principle, and thus we cannot apply the Barles-Souganidis convergence framework~\cite{BSnum}. However, we show that this new equation requires subsolutions to be unique (up to additive constants).  The uniqueness of subsolutions will be used to propose an alternate convergence framework in \autoref{sec:approx}.

\subsection{Viscosity subsolutions of modified PDE}\label{sec:PDE}

The examples of the previous section indicate that traditional arguments relying on the comparison principle will not be effective for establishing the convergence of approximation schemes for the second boundary value problem for the \MA equation.  We propose an alternate approach, which enforces the transport constraint~\eqref{eq:bvp2} not as a boundary condition, but rather as a condition that must be satisfied in the interior of the domain.  This will be accomplished by requiring the convex function $u$ to simultaneously be a subsolution of two different equations.

We will again make use of the defining function $H(y)$ for the target set (Definition~\ref{def:defining}).  Notice that the transport constraint~\eqref{eq:bvp2} is equivalent to the condition
\bq\label{eq:HJsub}
x\in X, \, y \in \partial u(x) \, \Rightarrow \, H(y) \leq 0.
\eq
Thus formally we expect the optimal transport potential to be a subsolution of the Hamilton-Jacobi equation
\bq\label{eq:HJ}
H(\nabla u(x)) = 0, \quad x \in X.
\eq

At the same time, we expect the convex function $u$ to be a solution, and therefore also a subsolution, of the \MA equation
\bq\label{eq:MA2}
-g(\nabla u(x))\det(D^2u(x)) + f(x) = 0, \quad x \in X.
\eq
A key observation we make is that if $u$ is a convex subsolution of both~\eqref{eq:HJ} and~\eqref{eq:MA2}, it is automatically a solution of~\eqref{eq:MA2}.  Intuitively, this is because strict subsolutions of the \MA equation generate ``too much'' mass:
\[ g(\nabla u(x))\det(D^2u(x)) > f(x). \]
On the other hand, strict subsolutions of the Hamilton-Jacobi equation constrain mass to be mapped inside the target set $Y$ and thus generates ``too little'' mass.  Enforcing both constraints simultaneously requires that the map $\nabla u$ produces an amount of mass that is ``just right'' and thus the \MA inequality is forced to attain equality.  See Figure~\ref{fig:massbalance}.

\begin{figure}
{\subfigure[]{\includegraphics[width=0.45\textwidth,clip=true,trim=0.5in 0 0.5in 0]{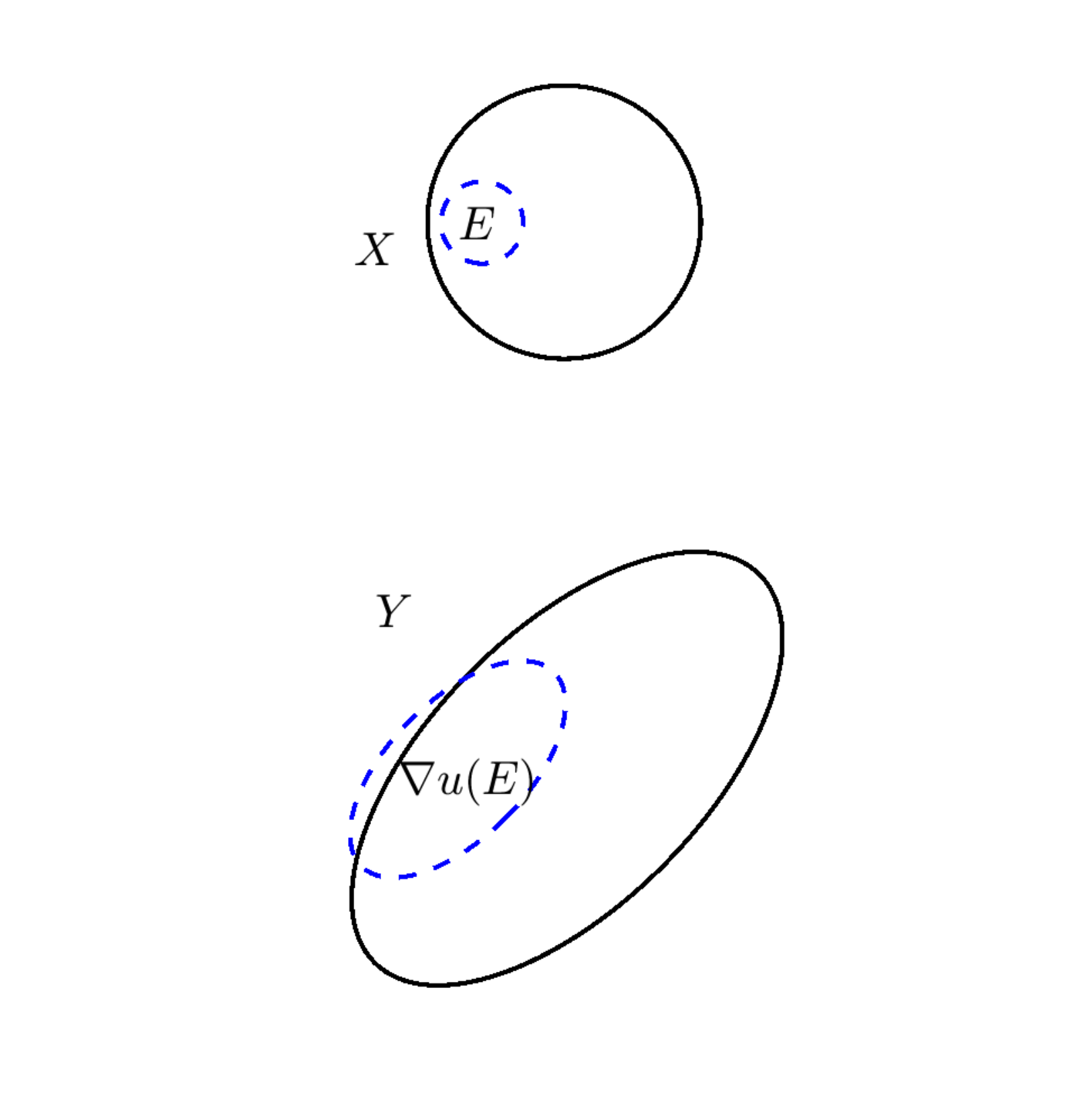}\label{fig:toomuch}}}
{\subfigure[]{\includegraphics[width=0.45\textwidth,clip=true,trim=0.5in 0 0.5in 0]{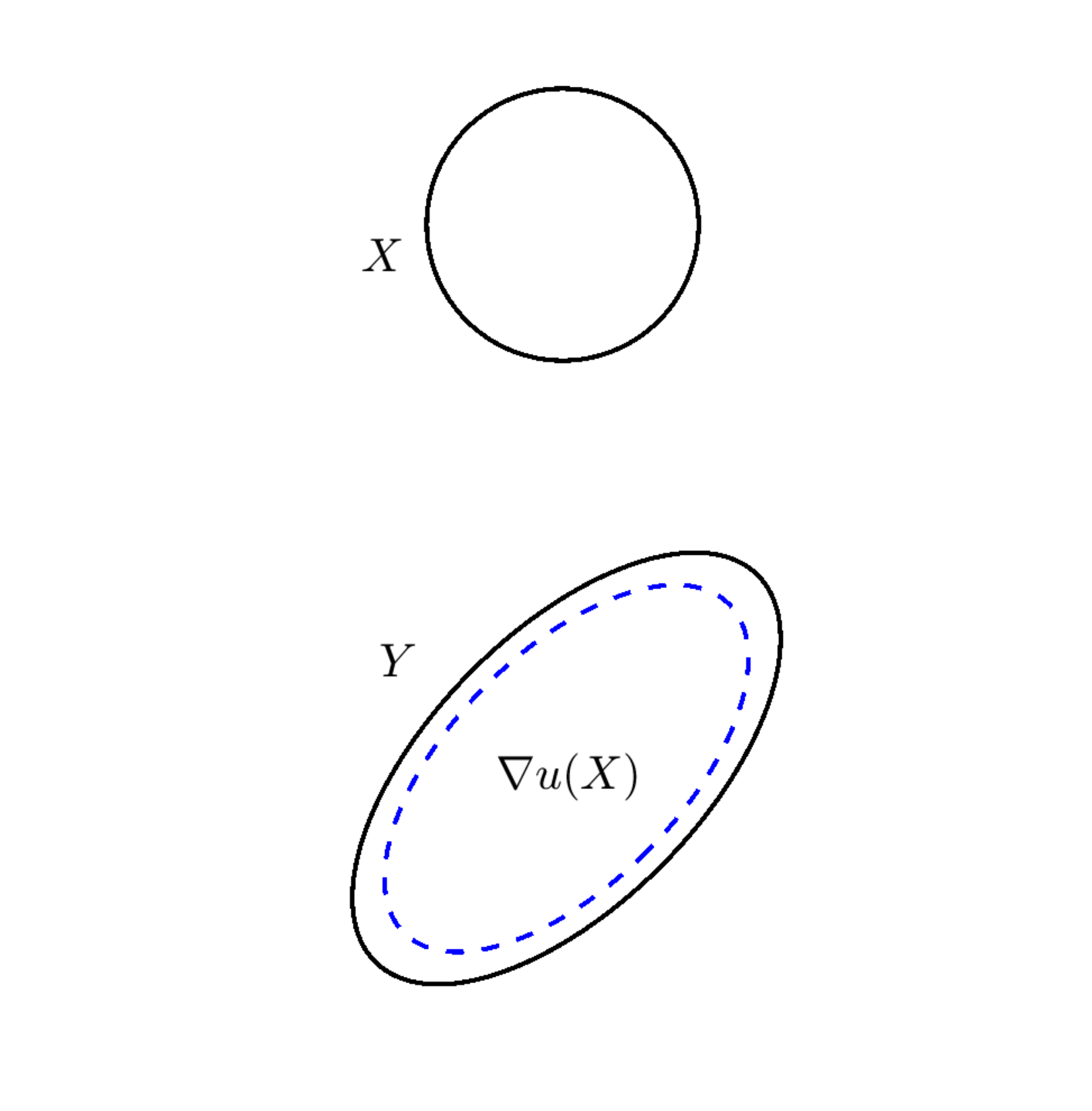}\label{fig:toolittle}}}
\caption{\subref{fig:toomuch}~Strict subsolutions of the \MA equation~\eqref{eq:MA2} satisfy $\int_{\partial u(E)}g(y)\,dy > \int_E f(x)\,dx$ and produce ``too much'' mass. \subref{fig:toolittle}~Strict subsolutions of the Hamilton-Jacobi equation~\eqref{eq:HJ} satisfy $\nabla u(X) \subset Y$ and produce ``too little'' mass.}
\label{fig:massbalance}
\end{figure}

There is one additional constraint we need to enforce: convexity of the solution $u$.  Formally, this means that the Hessian $D^2u(x)$ should be positive semi-definite and all eigenvalues of the Hessian non-negative.  If $\lambda_1(D^2u(x))$ denotes the smallest eigenvalue of the Hessian, this condition can be enforced by requiring $u$ to be a subsolution of the convex envelope equation described in~\cite{ObermanCE}
\bq\label{eq:convex}
-\lambda_1(D^2u(x)) = 0.
\eq
We remark that in regions with zero mass ($f=0$), this is formally equivalent to the equations used to produce a local representation of the minimal convex extension described in~\cite{BenamouDuval_MABVP2}.

Thus we propose looking for a function $u$ that is a viscosity subsolution of three different PDEs: the \MA equation~\eqref{eq:MA}, the Hamilton-Jacobi equation~\eqref{eq:HJ}, and the convex envelope equation~\eqref{eq:convex}.  These requirements can be satisfied by seeking a viscosity subsolution of the single combined PDE
\bq\label{eq:MAOT}
\max\left\{-g(\nabla u(x))\det(D^2u(x))+f(x), -\lambda_1(D^2u(x)), H(\nabla u(x))\right\} = 0, \quad x\in X.
\eq

The next subsection is devoted to proving the following key theorem.
\begin{theorem}[Viscosity subsolutions solve OT problem]\label{thm:subsolutions}
Let $u:X\to\R$ be an upper semi-continuous viscosity subsolution of~\eqref{eq:MAOT} where the data satisfies Hypothesis~\ref{hyp}.  Then $u$ is uniquely defined on $supp(f)$ up to additive constants and the subgradient map $\partial u$ solves the optimal transport problem.
\end{theorem}

\begin{remark}
The proof that a viscosity subsolution exists is deferred to sections~\ref{sec:approx}-\ref{sec:construction}, where suitable approximation schemes are used to constructively demonstrate existence.
\end{remark}

\subsection{Equivalence with optimal transport problem}\label{sec:equivalence}

We now build up several lemmas that will be used to prove Theorem~\ref{thm:subsolutions}.

We first note that since we are only interested in viscosity subsolutions, we only need to compute the lower envelope $F_*$ of the PDE operator in Definition~\ref{def:viscosity}.  This can be simplified because of the semi-continuity of the data $f$, $g$.
\begin{lemma}
Let $u$ be an upper semi-continuous subsolution of~\eqref{eq:MAOT} where the data satisfies Hypothesis~\ref{hyp}.  Suppose also that for some $\phi\in C^2$, the function $u-\phi$ has a strict local maximum at $x\in X$.  Then 
\[ \max\{-g(\nabla \phi(x))\det(D^2\phi(x))+f(x), -\lambda_1(D^2\phi(x)), H(\nabla\phi(x))\} \leq 0. \]  
\end{lemma}
\begin{proof}
From Definition~\ref{def:viscosity}, we must have that
\[ -\lambda_1(D^2\phi(x)) \leq F_*(x,u(x),\nabla\phi(x),D^2\phi(x)) \leq 0. \]
Thus $\det(D^2\phi(x)) \geq 0$ and the result follows immediately from the lower semi-continuity of $f$, the upper semi-continuity of $g$, and the continuity of $H$.
\end{proof}

The convexity constraint is automatically enforced by subsolutions, as demonstrated in the previous works~\cite[Lemma~3.4]{Hamfeldt_Gauss} and~\cite[Theorem~1]{ObermanCE}.
\begin{lemma}[Subsolutions are convex]~\label{lem:convex}
Let $u:{X}\to\R$ be an upper semi-continuous subsolution of~\eqref{eq:MAOT}.  Then $u$ is convex.
\end{lemma}

Next we demonstrate that subgradient maps induced by subsolutions do transport all mass into the target set.
\begin{lemma}[Subsolutions satisfy optimal transport constraint]\label{lem:OTconstraint}
Let $u:{X}\to\R$ be an upper semi-continuous subsolution of~\eqref{eq:MAOT}.  Then $\partial u(X) \subset \bar{Y}$.
\end{lemma}
\begin{proof}
Consider any $x_0\in X$.  From Lemma~\ref{lem:convex}, $u$ is convex and therefore the subgradient image $\partial u(x_0)$ is a convex set.  Let $p$ be an extreme point of this convex set.  That is, if $L$ is any open line segment connecting two distinct points in $\partial u(x_0)$, $p$ does not lie in $L$.  Additionally, all points in a closed, bounded, and finite dimensional convex set can be expressed as a convex combination of the extreme points of the set.

Now choose any $\epsilon>0$.  Since $p$ is an extreme point, there exists $y \in X$ such that $u$ is differentiable at $y$ and $\abs{\nabla u(y) - p} < \epsilon$~\cite[Corollary~2.5.3]{Borwein_convex}.  Moreover, by a continuity property of the (set-valued) subgradient of a convex function, there exists $\delta >0$ such that $\abs{\partial u(z)-\nabla u(y)} < \epsilon$ whenever $\abs{z-y}<\delta$~\cite[Exercise~2.2.22]{Borwein_convex}.

By Alexandrov's Theorem~\cite{Alexandrov}, there exists some $z\in B(y,\delta)$ such that $u$ has a second-order Taylor expansion at $z$.  That is, there exists $q\in\partial u(z)$ and a symmetric positive semi-definite matrix $A$ such that
\[ u(x) = u(z) + q\cdot(z-x) + (z-x)^TA(z-x) + o(\abs{z-x}^2). \]
Moreover, we have
\[ \abs{q-p} \leq \abs{q-\nabla u(y)} + \abs{\nabla u(y) - p} < 2\epsilon. \]

Define the smooth test function
\[ \phi(x) = u(z) + q\cdot(z-x) + (z-x)^T(A+I)(z-x) \]
and note that $u-\phi$ has a local maximum at $x=z$ with $\nabla\phi(z)=q$ and $D^2\phi(z)=A+I$. Then since $u$ is a subsolution of~\eqref{eq:MAOT} we must have
\[
H(q) \leq \max\{-g(q)\det(A+I)+f(z),-\lambda_1(A+I), H(q)\} \leq 0
\]
and thus $q \subset\bar{Y}$ from Definition~\ref{def:defining}.  We then have that
\[ \text{dist}(p,\bar{Y}) < 2\epsilon. \]
Taking $\epsilon\to0$ we obtain $p \in \bar{Y}$.

Since this holds for all extreme points of the convex set $\partial u(x_0)$ and since $Y$ is itself convex, we conclude that $\partial u(x_0) \subset \bar{Y}$ for any $x_0 \in X$ and thus $\partial u(X) \subset \bar{Y}$.
\end{proof}

We now turn our attention to demonstrating that subsolutions are actually solutions of the \MA equation.  
We begin by making rigorous our intuitive idea that subgradient maps induced by subsolutions generate ``to much'' mass.

\begin{lemma}[Subsolutions and mass balance]\label{lem:subMass}
Let $u:{X}\to\R$ be an upper semi-continuous subsolution of~\eqref{eq:MAOT} and let $E$ be any uniformly convex set such that $\bar{E}\subset X$.  Then 
\bq\label{eq:subMass}
\int_{\partial u(E)}g(y)\,dy \geq \int_E f(x)\,dx.
\eq
\end{lemma}

\begin{proof}
From Lemma~\ref{lem:convex}, $u$ is convex and therefore continuous on $\bar{E}\subset X$.  Moreover, from Lemma~\ref{lem:OTconstraint}, $\partial u(E) \subset \partial u(X) \subset \bar{Y}$.  

Now we choose $\epsilon>0$ and mollify the data in order to take advantage of existing results on weak solutions of the \MA equation.  

We define the set
\[ E_\epsilon = \{x\in E \mid \dist(x,\partial E)>\epsilon\}. \]
By Baire's Theorem and the integrability of $f$, we can produce a locally Lipschitz continuous approximation $\tilde{f}_\epsilon$ that underestimates the original lower semi-continuous density function $f$, $0 \leq \tilde{f}_\epsilon \leq f$, with $\|f - \tilde{f}_\epsilon\|_{L^1(X)} < \epsilon$~\cite[Exercise~1.7.15(c)]{Engelking_Topology}.  We further adjust this in an $\epsilon$-neighbourhood of the boundary by defining
\[ f_\epsilon(x) = \frac{1}{\epsilon}\tilde{f}_\epsilon(x) \min\{\text{dist}(x,\partial E), \epsilon\}. \]
Notice that $f_\epsilon \leq \tilde{f}_\epsilon \leq f$ in $\bar{E}$ with $f_\epsilon = \tilde{f}_\epsilon$ in $E_\epsilon$.

Next we define the augmented target set
\[ Y_\epsilon = Y \cup \{y\in\R^n \mid \text{dist}(y,\partial Y) < \epsilon\} \]
and produce a locally Lipschitz continuous approximation $\tilde{g}_\epsilon$ that overestimates the original upper semi-continuous density function $g$, $\tilde{g}_\epsilon \geq g$ on $Y_\epsilon$, with $\tilde{g}_\epsilon = 0$ on $Y_{\epsilon}^c$ and $\|g - \tilde{g}_\epsilon\|_{L^1(\R^n)} < \epsilon$.  We further shift this so that it is strictly positive by defining
\[ g_\epsilon(y) = \tilde{g}_\epsilon(y) + \epsilon. \]

Now we consider the following Dirichlet problem for the \MA equation.
\bq\label{eq:MADirichlet}
\begin{cases}
-g_\epsilon(\nabla u_\epsilon(x))\det(D^2u_\epsilon(x)) + f_\epsilon(x) = 0, & x \in E\\
u_\epsilon(x) = u(x), & x \in \partial E\\
u_\epsilon \text{ is convex.}
\end{cases}
\eq

First we notice that $u$ itself is a viscosity subsolution of this PDE.  To check this, consider any $x_0\in E$ and convex $\phi\in C^2$ such that $u-\phi$ has a local maximum at $x_0$. Using the fact that $u$ is a subsolution of~\eqref{eq:MAOT} we can verify that
\begin{align*}
-g_\epsilon(\nabla & \phi(x_0))\det(D^2\phi(x_0)) + f_\epsilon(x_0) \leq -g(\nabla \phi(x_0))\det(D^2\phi(x_0)) + f(x_0)\\
  &\leq \max\left\{-g(\nabla \phi(x_0))\det(D^2\phi(x_0)) + f(x_0), -\lambda_1(D^2\phi(x_0)), H(\nabla\phi(x_0))\right\}\\
	&\leq 0.
\end{align*}
Thus $u$ is a subsolution.

Next, we let $v_\epsilon$ be any continuous Aleksandrov solution of~\eqref{eq:MADirichlet}; see Theorem~\ref{thm:aleksDirichlet}.  By Theorem~\ref{thm:aleksVisc}, $v_\epsilon$ is also a viscosity solution, and therefore a supersolution as well.  We can then use the classical comparison principle (Theorem~\ref{thm:comparisonMA}) to conclude that $u \leq v_\epsilon$ in $E$ with equality on the boundary $\partial E$.  Then we also have that $\partial v_\epsilon(E) \subset\partial u(E) \subset \bar{Y}$ by~\cite[Lemma 1.4.1]{Gutierrez}.

Combining these observations with the fact that $v_\epsilon$ is an Aleksandrov solution of~\eqref{eq:MADirichlet}, we can calculate
\begin{align*}
\int_{\partial u(E)} g(y)\,dy &\geq \int_{\partial v_\epsilon(E)} g(y)\,dy\\
  &\geq \int_{\partial v_{\epsilon}(E)} \tilde{g}_\epsilon(y)\,dy - \|g-\tilde{g}_\epsilon\|_{L_1(\partial v_{\epsilon}(E))}\\
	&\geq \int_{\partial v_{\epsilon}(E)} \left(g_{\epsilon}(y)-\epsilon\right)\,dy - \epsilon\\
	&= \int_E f_{\epsilon}(x)\,dx - \epsilon\left(\abs{\partial v_{\epsilon}(E)}+1\right)\\
	&\geq \int_{E_{\epsilon}} f_{\epsilon}(x)\,dx - \epsilon (\abs{Y}+1)\\
	&= \int_{E_\epsilon}\tilde{f}_\epsilon(x)\,dx - \epsilon(\abs{Y}+1)\\
	&\geq \int_{E_\epsilon} f(x)\,dx - \|\tilde{f}_\epsilon - f\|_{L_1(E_\epsilon)} - \epsilon(\abs{Y}+1)\\
	&\geq \int_{E_\epsilon} f(x)\,dx - \epsilon(\abs{Y}+2).
\end{align*}
 
Since $f\in L^1(E)$ we can take $\epsilon\to0$ to obtain
\[ \int_{\partial u(E)} g(y)\,dy \geq \int_E f(x)\,dx.  \]
\end{proof}

This result immediately generalises to general subsets of the domain through approximation by the union of convex sets as in~\cite[Proposition~1.7.1]{Gutierrez}.
\begin{corollary}\label{cor:subMass}
Let $u:{X}\to\R$ be an upper semi-continuous subsolution of~\eqref{eq:MAOT} and let $E\subset X$.  Then 
\[ \int_{\partial u(E)}g(y)\,dy \geq \int_E f(x)\,dx. \]
\end{corollary}

\begin{lemma}[Viscosity subsolutions are Aleksandrov solutions]\label{lem:viscAleks}
Let $u:{X}\to\R$ be an upper semicontinuous subsolution of~\eqref{eq:MAOT}.  Then $u$ is an Aleksandrov solution of~\eqref{eq:MA}.
\end{lemma}
\begin{proof}
Choose any $E\subset X$. Then from the mass balance condition on the data, Corollary~\ref{cor:subMass}, and Lemma~\ref{lem:OTconstraint} we can compute
\begin{align*}
\int_E f(x)\,dx &= \int_X f(x)\,dx - \int_{X\backslash E}f(x)\,dx\\
  &\geq \int_Y g(y)\,dy - \int_{\partial u(X\backslash E)} g(y)\,dy\\
	&\geq \int_{\partial u(X)} g(y)\,dy - \int_{\partial u(X\backslash E)} g(y)\,dy.
\end{align*} 
Since the intersection of $\partial u(E)$ and $\partial u(X\backslash E)$ has zero measure~\cite[Lemma~1.1.12]{Gutierrez} and $g$ does not give mass to small sets, this simplifies to
\[ \int_E f(x)\,dx \geq \int_{\partial u(E)}g(y)\,dy. \]
Combining this with Corollary~\ref{cor:subMass}, we obtain 
\[ \int_E f(x)\,dx = \int_{\partial u(E)}g(y)\,dy \]
and thus $u$ is an Aleksandrov solution.
\end{proof}

These lemmas lead directly to our main result.
\begin{proof}[Proof of Theorem~\ref{thm:subsolutions}]
This follows immediately from Lemmas~\ref{lem:convex},~\ref{lem:OTconstraint},~\ref{lem:viscAleks} and Theorem~\ref{thm:aleksExist}.
\end{proof}
%
%

\section{Approximation of solutions}\label{sec:approx}
We are particularly interested in developing criterion that will allow for convergent numerical approximation of the second boundary value problem for the \MA equation.  The Barles-Souganidis framework suggests that consistent, monotone schemes will converge.  However, their proof requires a comparison principle, which is not satisfied by this PDE (\autoref{sec:comparison}).

In this section, we show that consistent, monotone approximations can indeed be used to approximate~\eqref{eq:MAOT} if they are additionally required (or modified) to be \emph{under-estimating}.  A surprising result of this section is that a wide range of artificial boundary conditions will lead to correct results.

A delicate issue is the existence of a solution to the approximation scheme.  Recall that the PDE itself has only one subsolution (up to additive constants).  Thus perturbing the equation could easily lead to a scheme with no subsolutions, and thus no solution.  We show that by restricting ourselves to schemes that underestimate the value of the PDE operator, we can ensure existence and stability of the methods.  

%
%
%

\subsection{Properties of schemes}\label{sec:schemes}
We begin by reviewing basic properties of approximation schemes.

Consider a set of discretisation points $\G^h\subset\bar{X}$, which can contain points in both the domain $X$ and its boundary.  We define the resolution $h$ by
\bq\label{eq:h}
h = \sup\limits_{x \in X} \inf\limits_{y \in \G^h} \abs{x-y}.
\eq
That is, every ball of radius $h$ contains a discretisation point.

We consider finite difference schemes that have the form
\bq\label{eq:approx} F^h(x,u(x),u(x)-u(\cdot)) = 0, \quad x\in \G^h \eq
where $u:\G^h\to\R$ is a grid function.  We are interested in schemes that are consistent and monotone.  

We require a one-sided notion of consistency when the scheme acts on convex functions.  Since the boundary condition used is artificial and we are only seeking convergence in the interior of the domain, we can make use of a separate definition of consistency at boundary points.
\begin{definition}[Consistency]\label{def:consistency}
The scheme~\eqref{eq:approx} is \emph{consistent} with~\eqref{eq:MAOT} on $X$
 if for any smooth convex function $\phi$ and $x\in X$,
\[ \liminf_{h\to0,y\in\G^h\to x,\xi\to0} F^h(y,\phi(y)+\xi,\phi(y)-\phi(\cdot)) \geq F(x,\phi(x),\nabla\phi(x),D^2\phi(x)). \]
The scheme is \emph{uniformly consistent} with~\eqref{eq:MAOT} if it is consistent and the above limit is achieved uniformly on $X$.
\end{definition}
Note that this is consistent with conventional definitions of consistency since we will also require schemes to underestimate the value of the PDE operator in order to prove existence and stability.

We also rely on the usual notion of monotonicity.  
\begin{definition}[Monotonicity]\label{def:monotonicity}
The scheme~\eqref{eq:approx} is \emph{monotone} if $F^h$ is a non-decreasing function of its final two arguments.
\end{definition}

We remark that boundary points need to be available in order to build schemes for second-order elliptic equations that are consistent and monotone at interior points near the boundary~\cite{Kocan}.  While the particular definition of the scheme $F^h$ at the boundary does not particularly matter (i.e. a variety of boundary conditions are possible), it needs to satisfy the monotonicity condition.  In the following, we will consider a simple homogeneous Dirichlet condition
\bq\label{eq:Dirichlet}
F^h(x,u(x),u(x)-u(\cdot)) = u(x), \quad x \in \partial X.
\eq
The following analysis can also be extended to more general Dirichlet boundary conditions and some nonlinear Neumann type boundary conditions.

Schemes are required to satisfy mild stability and continuity conditions.
\begin{definition}[Stability]\label{def:stability}
The scheme~\eqref{eq:approx} is \emph{stable} if there exists a constant $M$, independent of $h$, such that if $h>0$ and $u^h$ is any solution of~\eqref{eq:approx} then $\|u^h\|_\infty \leq M$.
\end{definition}

\begin{definition}[Continuity]\label{def:continuous}
The scheme~\eqref{eq:approx} is \emph{continuous} if $F^h$ is continuous in its last two arguments.
\end{definition}
We remark that we do not require any form of uniform continuity, and the limiting PDE operator need not be continuous.

Even though the original PDE does not have a comparison principle, the discrete approximation does satisfy a weakened form of the comparison principle.
\begin{lemma}[Discrete comparison principle~{\cite[Theorem~5]{ObermanSINUM}}]\label{lem:discreteComp}
Let $F^h$ be a monotone scheme and $F^h(x,u(x),u(x)-u(\cdot)) < F^h(x,v(x),v(x)-v(\cdot))$ for every $x\in\G^h$.  Then $u(x) \leq v(x)$ for every $x\in\G^h$.
\end{lemma}

\begin{remark}
Because the inequality in this discrete comparison principle is strict, it does not guarantee solution uniqueness.  For some monotone schemes, it is not possible to find $u, v$ such that $F^h[u] < F^h[v]$ at every grid point.
\end{remark}

Finally, we will use strict sub- and supersolutions of schemes to demonstrate existence and stability.
\begin{definition}[Subsolutions and supersolutions of schemes]\label{def:subsuper}
A grid function $u$ is a \emph{strict sub(super) solution} of the scheme~$F^h$ if 
\[ F^h(x,u(x),u(x)-u(\cdot)) < (>) 0 \]
for every $x\in\G^h$. 
\end{definition}

Strict supersolutions are typically straightforward to generate when approximating~\eqref{eq:MAOT}.  Strict subsolutions are more delicate; this is related to the fact that the PDE itself has many supersolutions but only a single subsolution. A variety of options are possible for ensuring the existence of a strict subsolution.  
A simple option is to require the scheme to underestimate the value of the PDE operator on an appropriate class of test functions.

\begin{definition}[Underestimating scheme]\label{def:underestimate}
Let $\G^h \subset \bar{X}$ be a set of discretisation points, $V_x^h$ the Voronoi cell containing $x\in\G^h$, and $A_x^h$ the area of this Voronoi cell.
The approximation scheme~\eqref{eq:approx} is an \emph{underestimating scheme} if, whenever $u$ is a convex piecewise linear surface with node values at $\G^h$ satisfying $\partial u(x) \cap \bar{Y} \neq \emptyset$ for every $x\in \G^h$, then
\[ F^h(x,u(x),u(x)-u(\cdot)) \leq \max\left\{\frac{-\int_{\partial u(V_x^h)}g(y)\,dy + \int_{V_x^h}f(z)\,dz}{A^h_x}, 0\right\} \]
for every $x\in\G^h\cap X$.

We say that the approximation scheme is a \emph{strictly underestimating scheme} if this inequality is strict.
\end{definition}

We note that in this definition, underestimation is required on piecewise linear functions, which concentrate all mass on the grid points $x\in\G^h$.  Thus $\partial u(x) = \partial u(V_x^h)$ at these points since each Voronoi cell $V_x^h$ contains only a single node point $x$.  The expressions in this definition are reminiscent of the methods~\cite{Aurenhammer_minkowski,olikerprussner88} for semidiscrete optimal transport, which directly evaluate $\int_{\partial u(V_x^h)}h\,dy$  for $u$ in this same class of functions.

\begin{remark}
An alternative definition for an underestimating scheme is to require 
\[ F^h(x,u(x),u(x)-u(\cdot)) \leq \max\{-g(\nabla u(x))\det(D^2u(x))+f(x), -\lambda_1(D^2u(x)), H(\nabla u(x))\} \]
for every smooth convex $u$ and $x\in\G^h\cap X$.
\end{remark}



In \autoref{sec:stability}-\ref{sec:convergence}, we will focus on proving the following two key convergence theorems. In \autoref{sec:construction} we give examples of how to construct schemes that satisfy these hypotheses.
\begin{theorem}[Existence and stability]\label{thm:stability}
Given data satisfying Hypothesis~\ref{hyp}, let $F^h$ be a uniformly consistent, monotone, continuous, strictly underestimating approximation scheme.  Then the approximation scheme~\eqref{eq:approx} has a solution $u^h$ and the scheme is stable.
\end{theorem}

\begin{theorem}[Convergence of schemes]\label{thm:converge}
Given data satisfying Hypothesis~\ref{hyp}, let $u^h$ be any solution of the consistent, monotone, stable approximation scheme~\eqref{eq:approx}.  Define the piecewise constant nearest-neighbours extension
\bq\label{eq:extension} U^h(x) = \sup\left\{u^h(y) \mid y\in\G^h, \, \abs{y-x} = \min\limits_{z\in\G^h}\abs{z-x}\right\}. \eq
Then for any $x\in X$,
\[ \lim\limits_{h\to0}U^h(x) = u(x) \]
where $u(x)$ is an Aleksandrov solution of~\eqref{eq:MA}, \eqref{eq:bvp2}.
\end{theorem}

\subsection{Existence and stability}\label{sec:stability}
First we seek to prove that approximation schemes have solutions that can be bounded in $L^\infty$.  The key to existence is to use a discrete version of Perron's method.  Solutions are then bounded via the discrete comparison principle (Lemma~\ref{lem:discreteComp}).
Both arguments will require the existence of strict sub- and supersolutions. 

\begin{lemma}[Existence of stable subsolution]\label{lem:existSubUnder}
Let $F^h$ be a strictly underestimating approximation scheme for~\eqref{eq:MAOT} with data satisfying Hypothesis~\ref{hyp}.  Then there exists a bounded function $u:\bar{X}\to\R$, independent of $h$, such that $u$ is a strict subsolution of the scheme~\eqref{eq:approx}.
\end{lemma}

\begin{proof}
We begin by approximating the density function $f$ by a discrete measure,
\[ \mu(x) = \sum\limits_{x_i\in\G^h}\mu_i\delta_{x_i}(x), \quad \mu_i = \int_{V_{x_i}^h}f(z)\,dz. \]
Now we let $u$ be a convex Aleksandrov solution of the \MA equation
\bq\label{eq:MAdiscrete}
\begin{cases}
\int_{\partial u(E)}g(p)\,dp = \mu(E)\\
\partial u(\R^n) \subset \bar{Y}.
\end{cases}
\eq
Since $u$ is unique up to additive constants and bounded on compact sets, we can assume without loss of generality that $M < u(x) < 0$ for some $M\in\R$ and every $x\in\bar{X}$.

Next we define $v(x)$ to be the largest convex function such that $v(x)=u(x)$ for every $x\in\G^h$.  Notice that $v$ is piecewise linear with nodes at the points $x\in\G^h$.  Moreover, $v \geq u$ by definition since $u$ is itself a convex function, with $M < v(x) < 0$ for $x\in\bar{X}$.  Because of the equality at the node points, we also have $\partial u(x) \subset \partial v(x)$ for every $x\in\G^h$.  In particular, since $\partial u(\R^n) \subset \bar{Y}$, we can conclude that $\partial v(x) \cap \bar{Y} \neq \emptyset$ for every $x\in\G^h$.

Now we show that $v$ (restricted to node points) is a strict subsolution of $F^h$.

Choose any $x\in\G^h$ and $p\in\partial u(x)\subset\partial v(x)$.  If $x\in\partial X$ then
\[ F^h(x,v(x),v(x)-v(\cdot)) = v(x) < 0 \]
and the condition for a subsolution is trivially satisfied.

Otherwise, we have $x\in\G^h\cap X$.

{\bf Case 1}: \[\max\left\{\dfrac{-\int_{\partial v(V_x^h)}g(y)\,dy + \int_{V_x^h}f(z)\,dz}{A^h_x}, 0\right\} = 0.\]
In that case, since the scheme underestimates it will trivially satisfy
\[ F^h(x,v(x),v(x)-v(\cdot)) < 0. \]


{\bf Case 2}: \[\max\left\{\dfrac{-\int_{\partial v(V_x^h)}g(y)\,dy + \int_{V_x^h}f(z)\,dz}{A^h_x}, 0\right\} = \dfrac{-\int_{\partial v(V_x^h)}g(y)\,dy + \int_{V_x^h}f(z)\,dz}{A^h_x}.\]
Since the scheme underestimates, we have
\[
A_{x}^hF^h(x,v(x),v(x)-v(\cdot)) < -\int_{\partial v(V_x^h)}g(y)\,dy + \int_{V_x^h}f(z)\,dz.
\]
Since $\partial u(x) \subset \partial v(x)$ and since both measures concentrate all mass at the node points, we obtain
\begin{align*}
A_{x}^hF^h(x,v(x),v(x)-v(\cdot)) &< -\int_{\partial u(V_x^h)}g(y)\,dy + \int_{V_x^h}f(z)\,dz\\
  &= -\mu(V_x^h) + \int_{V_x^h}f(z)\,dz\\
	&= -\int_{V_x^h}f(z)\,dz+\int_{V_x^h}f(z)\,dz\\
	&=0.
\end{align*}

We conclude that $v$ is a stable strict subsolution of $F^h$.
\end{proof}

\begin{lemma}[Existence of stable supersolution]\label{lem:super}
Let $F^h$ be a uniformly consistent approximation scheme for~\eqref{eq:MAOT} with data satisfying Hypothesis~\ref{hyp}.  Then there exists a bounded function $u:\bar{X}\to\R$, independent of $h$, such that $u$ is a strict supersolution of the scheme~\eqref{eq:approx} for sufficiently small $h>0$.
\end{lemma}

\begin{proof}
Choose any $p\notin \bar{Y}$ so that $H(p)>\epsilon$ for some $\epsilon>0$.
Define the function
\[ u(x) = p\cdot x -\inf\limits_{y\in\partial X}\{p\cdot y\}  + \epsilon. \]

For $x\in X$, the PDE operator~\eqref{eq:MAOT} satisfies
\[ F(x,u(x),\nabla u(x), D^2u(x)) \geq H(\nabla u(x)) > \epsilon. \]
Since $F^h$ is unifromly consistent we have for $x\in \G^h\cap X$ and sufficiently small $h>0$ (independent of $x$),
\begin{align*} F^h(x,u(x),u(x)-u(\cdot)) 
&\geq \liminf_{h\to0,y\in\G^h\to x,\xi\to0} F^h(y,u(y)+\xi,u(y)-u(\cdot)) - \frac{\epsilon}{2} \\
&\geq F(x,u(x),\nabla u(x),D^2u(x)) - \frac{\epsilon}{2} \\
&>\frac{\epsilon}{2}.
\end{align*}

For $x\in \partial X$, the boundary operator~\eqref{eq:Dirichlet} satisfies
\[ F^h(x,u(x),u(x)-u(\cdot)) = u(x) \geq \epsilon.\]

We conclude that $u$ is a strict supersolution.
\end{proof}

%
%
%

\begin{lemma}[Existence of solution]\label{lem:exist}
Under the assumptions of Hypothesis~\ref{hyp}, let $F^h$ be a consistent, monotone, continuous, strictly underestimating approximation scheme.  Then there exists a solution $u^h$ to the scheme~\eqref{eq:approx}.
\end{lemma}
\begin{proof}
By Lemmas~\ref{lem:existSubUnder}-\ref{lem:super} there exist a strict subsolution $v^h$ and a strict supersolution $w^h$.  Using a discrete version of Perron's method, identical to the proof of~\cite[Lemma~5.8]{Hamfeldt_Gauss}, we can construct a solution to the scheme.
\end{proof}

\begin{lemma}[Stability of solutions]\label{lem:stability}
Under the assumptions of Hypothesis~\ref{hyp}, let $F^h$ be a consistent, monotone, continuous, strictly underestimating approximation scheme.  Then $F^h$ is stable for sufficiently small $h>0$.
\end{lemma}
\begin{proof}
Let $u^h$ be any solution of~\eqref{eq:approx}.  For small enough $h>0$, Lemmas~\ref{lem:existSubUnder}-\ref{lem:super} ensure the existence of bounded, strict sub- and supersolutions $v^h$ and $w^h$.  Note also that for any $x\in\G^h$,
\begin{align*} F^h(x,&v^h(x),v^h(x)-v^h(\cdot)) < 0 \\ &= F^h(x,u^h(x),u^h(x)-u^h(\cdot))\\ &< F^h(x,w^h(x),w^h(x)-w^h(\cdot)). \end{align*}
By the discrete comparison principle (Lemma~\ref{lem:discreteComp}), we have
\[ -M \leq v^h(x) \leq u^h(x) \leq w^h(x) \leq M \]
for some $M\in\R$ independent of $h$.

Thus $u^h$ is bounded independent of $h$ and the scheme is stable.
\end{proof}

These lemmas immediately yield the proof of Theorem~\ref{thm:stability}.

\subsection{Convergence of schemes}\label{sec:convergence}
Next we prove the main converge result (Theorem~\ref{thm:converge}).  The idea will be to use the usual Barles-Souganidis approach to demonstrate that the upper semi-continuous envelope of the approximate solutions is a subsolution of the PDE.  We will use the under-estimating property of schemes to control the lower semi-continuous envelope and complete the convergence proof.

\begin{lemma}[Upper semi-continuous envelope of approximations]\label{lem:uscApprox}
Under the assumptions of Hypothesis~\ref{hyp}, let $F^h$ be a uniformly consistent, monotone, continuous, strictly underestimating approximation scheme. Let $u^h$ be a solution of the scheme~\eqref{eq:approx} and $U^h$ a piecewise constant extension onto $\bar{X}$.  Define
\bq\label{eq:usc}
\bar{u}(x) = \limsup\limits_{h\to0, y\to x}U^h(y).
\eq
Then $\bar{u}(x)$ is a subsolution of~\eqref{eq:MAOT} on $X$.
\end{lemma}
\begin{proof}
This proof is identical to the Barles-Souganidis convergence result~\cite{BSnum}, and the extension to general point clouds in~\cite[Theorem~9]{Hamfeldt_Gauss}.
\end{proof}

\begin{lemma}[Approximations are non-positive]\label{lem:negative}
Under the assumptions of Hypothesis~\ref{hyp}, let $F^h$ be a uniformly consistent, monotone, continuous, strictly underestimating approximation scheme. Let $u^h$ be any solution of the scheme~\eqref{eq:approx} and $U^h$ a piecewise constant extension onto $\bar{X}$.  Define
\bq\label{eq:usc2}
\bar{u}(x) = \limsup\limits_{h\to0, y\to x}U^h(y).
\eq
Then $\bar{u}(x) \leq 0$ on $\bar{X}$.
\end{lemma}
\begin{proof}
For $\epsilon>0$ define
\[ v_\epsilon(x) = -\epsilon \text{dist}(x, \partial X)^2 + 2\epsilon M \]
where
\[ M = \sup\limits_{x\in\partial X} \text{dist}(x, \partial X)^2. \]

For $x\in X$, we notice that
\[ F(x,v_\epsilon(x),\nabla v_\epsilon(x),D^2v_\epsilon(x)) \geq -\lambda_1(D^2v_\epsilon(x)) \geq 2\epsilon. \]
Uniform consistency ensures that for sufficiently small $h$, we have
\[ F^h(x,v_\epsilon(x),v_\epsilon(x)-v_\epsilon(\cdot)) > 0. \]
For $x\in\partial X$, we have
\[ F^h(x,v_\epsilon(x),v_\epsilon(x)-v_\epsilon(\cdot)) = v_\epsilon(x) \geq \epsilon M > 0. \]
Thus for small enough $h$ and any $x\in\bar{X}\cap\G^h$,
\[ F^h(x,v_\epsilon(x),v_\epsilon(x)-v_\epsilon(\cdot)) > 0 = F^h(x,u^h(x),u^h(x)-u^h(\cdot)). \]
By the discrete comparison principle (Lemma~\ref{lem:discreteComp}), 
$u^h(x) \leq v_\epsilon(x) \leq 2\epsilon M$
 in $\bar{X}$.  Taking $\epsilon\to 0$ we obtain
\[ \bar{u}(x) \leq 0.  \]
\end{proof}

\begin{proof}[Proof of Theorem~\ref{thm:converge}]
As before, we define
\bq\label{eq:usc3}
\bar{u}(x) = \limsup\limits_{h\to0, y\to x}U^h(y),
\eq
which is a subsolution of~\eqref{eq:MAOT} by Lemma~\ref{lem:uscApprox} and is therefore an Aleksandrov solution of the optimal transport problem~\eqref{eq:MA}, \eqref{eq:bvp2}.

We also let $v^h$ be the Aleksandrov solution of the semi-discrete \MA equation
\bq\label{eq:MAdiscrete2}
\begin{cases}
\int_{\partial v^h(E)}g(p)\,dp = \mu^h(E)\\
\partial v^h(\R^n) \subset \bar{Y}\\
\sup\limits_{x\in X}v^h(x) = -h
\end{cases}
\eq
where the measure $\mu$ is given by
\[ \mu^h(x) = \sum\limits_{x_i\in\G^h}\mu_i\delta_{x_i}(x), \quad \mu_i = \int_{V_{x_i}^h}f(z)\,dz. \]
Notice that $\mu^h$ converges weakly to the measure with density $f$ since each $V_{x_i}^h$ has diameter $\bO(h)$ and they together form a partition of $X$.  Moreover, the family of solutions $v^h(x)$ is uniformly Lipschitz continuous (with Lipschitz constant constrained to lie in the set $\bar{Y}$). Thus by the stability of the optimal transport problem $v^h(x)$ converges uniformly to $\bar{u}(x)$, the solution of the original optimal transport problem (\cite[Theorem~5.20]{Villani2}).

From the proof of Lemma~\ref{lem:existSubUnder}, $v^h$ is a strict subsolution of the approximation scheme~\eqref{eq:approx}.  Since
\[ F^h(x,v^h(x),v^h(x)-v^h(\cdot)) < 0 = F^h(x,u^h(x),u^h(x)-u^h(\cdot)),  \]
the discrete comparison principle (Lemma~\ref{lem:discreteComp}) ensures that $v^h(x) \leq u^h(x)$ at every $x\in\G^h$.

Next we define
\bq\label{eq:lsc}
\underline{u} = \liminf\limits_{h\to0, y\to x}U^h(y).
\eq

  Let $K$ be a bound on the Lipschitz constant for the family $\{v^h\}$.  For each $x\in X$ and $h>0$ we can find some $x_h\in\G^h$ such that $U^h(x) = u^h(x_h)$ with $x_h\to x$ as $h\to0$.  Thus we can compute
\begin{align*}
U^h(x) = u^h(x_h) \geq v^h(x_h) \geq v^h(x) - K\abs{x_h-x}.
\end{align*}
Letting $h\to0$ we obtain
\[ \underline{u}(x) \geq \lim\limits_{h\to0} v^h(x) = \bar{u}(x). \]

Moreover, by definition we must have $\underline{u}(x) \leq \bar{u}(x)$.  Combining these results, we conclude that
\[ \lim\limits_{h\to0} U^h(x) = \bar{u}(x), \]
which is a solution of the optimal transport problem.
\end{proof}

\section{Examples of schemes}\label{sec:construction}
A key contribution of our convergence theorem is that it provides a solid theoretical foundation to many existing methods after only slight modification.  To demonstrate the reasonableness of our assumptions, we provide a concrete example of one such method.

Note that in fact, \emph{any} consistent, monotone scheme $F^h$ can be modified to produce a scheme that is (at least locally) underestimating and therefore convergent.  This can be accomplished by using a scheme of the form $F^h-h^\alpha$ where $\alpha>0$ is related to the discretisation error of the method.  Thus with only slight modification, existing schemes for the \MA equation~\cite{benamou2014monotone,BenamouDuval_MABVP2,BFO_OTNum,FinlayOberman,FroeseMeshfreeEigs,FO_MATheory,HS_Quadtree,Nochetto_MAConverge}, eigenvalues of the Hessian~\cite{ObermanCE,ObermanEigenvalues}, and Hamilton-Jacobi equations (which are first-order and more well-developed) can be fit within this convergence framework.

Below we give an example of a method that is naturally underestimating, which allows for an improvement in the formal consistency error.  The foundation of this method is a careful combination of the Lattice Basis Reduction (LBR) method~\cite{benamou2014monotone} and constrained Lax-Friedrichs approximations.  The LBR method has been successfully used to solve the second boundary value problem for the \MA equation in~\cite{BenamouDuval_MABVP2}.  In that work, an underestimating property of the LBR scheme was used to prove convergence.  However, existence of solutions to the scheme was left open.  Here we show how the LBR scheme can be incorporated into our convergence framework by using modified Lax-Friedrichs approximations for the transportation constraint and density functions.  The result is a globally monotone and underestimating scheme, which guarantees both existence of solutions and convergence to the appropriate weak solution.

For simplicity, we restrict our attention to Lipschitz continuous densities $f, g$ on a square domain $X$ in two dimensions.  However, these can be generalised to more complicated problems using, for example, the techniques of~\cite{HS_Quadtree}.

\subsection{Finite difference approximations}

For this example, our grid $\G^h$ consists of a uniform Cartesian grid with spacing $h$, augmented by $\bO(h^{3/2})$ points uniformly distributed along the boundary of the domain.

We will break the scheme into three components, which will all be consistent, monotone, and underestimating:
\begin{align*}
F_1^h[u] &\approx -g(\nabla u(x))\det(D^2u(x)) + f(x)\\
F_2^h[u] &\approx -\lambda_1(D^2u(x)) \\
F_3^h[u] &\approx H(\nabla u(x)).
\end{align*}
Then a consistent, monotone, strictly underestimating scheme is
\bq\label{eq:example1}
F^h[u] = \begin{cases}
\max\{F_1^h[u],F_2^h[u],F_3^h[u]\} - h^\alpha, & x \in \G^h\cap X\\
u(x), & x\in \G^h\cap\partial X
\end{cases}
\eq
where $h^\alpha$ is chosen to be less than the discretisation error of the component schemes.

We will rely on the following standard finite difference operators on the Cartesian grid.  Here $e_j$ is a unit vector in the coordinate direction $x_j$.
\begin{align*}
\Dt_{x_j}^+u(x) &= \frac{u(x+he_j)-u(x)}{h}\\
\Dt_{x_j}^-u(x) &= \frac{u(x)-u(x-he_j)}{h}\\
\Dt_{x_j}^0u(x) &= \frac{u(x+he_j)-u(x-he_j)}{2h}\\
\Dt_{x_jx_j}u(x) &= \frac{u(x+he_j)+u(x-he_j)-2u(x)}{h^2}\\
\Delta^hu(x) &= \Dt_{x_1x_1}u(x) + \Dt_{x_2x_2}u(x).
\end{align*}

We denote by $L_f$, $L_g$, and $L_H$ the Lipschitz constants of $f$, $g$, and $H$ respectively.

\subsection{Monge-Amp\`ere operator}  As an example of a scheme for the \MA operator that is known to be monotone and underestimating, we consider a scheme based on the Lattice Basis Reduction method~\cite{benamou2014monotone}. This scheme has a natural interpretation in terms of measures, and was also used to discretise the \MA equation in the method of~\cite{BenamouDuval_MABVP2}. 

This operator relies on the concept of a superbase.
\begin{definition}[Superbase]\label{def:superbase}
A superbase of $\Z^2$ is a triplet $(e,e',e'')\in(\Z^2)^3$ satisfying $\det(e,e')=1$ and $e+e'+e''=0$.
\end{definition}

Like other monotone finite difference methods for fully nonlinear elliptic equations, this scheme relies on a wide stencil.  The superbases will define the grid directions that are utilised in the stencil.  We consider the following admissible set.
\begin{align}\label{eq:admissible}
\Af^h_x = &\{(e,e',e'') \text{ superbases of }\Z^2 \mid \\ &x\pm he,x\pm he', x\pm he''\in\G^h, \|e\|_\infty, \|e'\|_\infty, \|e''\|_\infty < \sqrt{N} \}. \nonumber
\end{align}

The scheme makes use of standard centred differencing in grid-aligned directions.  
\bq\label{eq:2ndderiv}
\Delta^h_{ee}u(x) = {u(x+eh)+u(x-eh)-2u(x)}.
\eq
Then the \MA operator can be approximated by
\bq\label{eq:MAApprox}
{\det}^h(D^2u)(x) = \min\limits_{(e,e',e'')\in\Af} G(\max\{\Delta^h_{ee}u(x),0\},\max\{\Delta^h_{e'e'}u(x),0\},\max\{\Delta^h_{e''e''}u(x),0\})
\eq
where
\[
G(a,b,c) = \begin{cases}
bc & a \geq b+c\\
ca & b \geq c+a\\
ab & c \geq a+b\\
\frac{1}{2}(ab+bc+ca)-\frac{1}{4}(a^2+b^2+c^2) & \text{otherwise}.
\end{cases}
\]

We use a Lax-Friedrichs approximation of the target density $g(\nabla u)$, which is constrained to be non-negative,
\bq\label{eq:gApprox}
g^h(\nabla u)(x) = \max\{g(\Dt^0_{x_1}u(x),\Dt^0_{x_2}u(x)) + hL_g\Delta^hu(x),0\},
\eq
and an underestimate of the source density $f$,
\bq\label{eq:fApprox}
f^h(x) = \max\left\{f(x)-\frac{1}{\sqrt{2}}hL_f, 0\right\}.
\eq

Then a monotone scheme for the combined \MA equation is
\bq\label{eq:approxF1}
F_1^h(x,u(x),u(x)-u(\cdot)) = -g^h(\nabla u)(x) {\det}^h(D^2u)(x) + f^h(x).
\eq

\begin{lemma}[Underestimation of Monge-Amp\`ere]\label{lem:underMA}
Let $u$ be a convex piecewise linear surface with node values on $\G^h$.  For any $x\in\G^h\cap X$ define the rectangle $R_x^h=x+[-h/2,h/2]^2$.  Then
\[ F^h_1(x,u(x),u(x)-u(\cdot)) \leq \frac{-\int_{\partial u(R_x^h)}g(y)\,dy + \int_{R_x^h}f(z)\,dz}{h^2}. \]
\end{lemma}
\begin{proof}
We begin by controlling the approximation of the source density function.  For any $z\in R_x^h$ we have
\[  f(z) \geq f(x) - L_f\abs{z-x} \geq f(x)-L_f\frac{h}{\sqrt{2}}. \]
Since $f$ is non-negative, we conclude that
\[ f(z) \geq \max\left\{f(x)-L_f\frac{h}{\sqrt{2}},0\right\} = f^h(x). \]
Therefore
\bq\label{eq:fUnder} f^h(x) \leq \frac{1}{h^2}\int_{R_x^h}f(z)\,dz. \eq

Next we control the target density.  We begin by investigating arguments belonging to $\partial u(x)$.  Define $v$ as the following convex cone: 
\[ v(x) = \sup\left\{v \text{ convex}\mid v(x)=u(x),\, v(x\pm h e_1)=u(x\pm he_1), \, v(x\pm he_2) = u(x \pm he_2) \right\}. \]
Notice that $v \geq u$ and thus $\partial u(x) \subset \partial v(x)$ where $\partial v(x)$ is given explicitly by the following rectangle.
\begin{align*} R \equiv \partial v(x) &= \left[\Dt_{x_1}^-v(x),\Dt_{x_1}^+v(x)\right]\times\left[\Dt_{x_2}^-v(x),\Dt_{x_2}^+v(x)\right]\\
&=\left[\Dt_{x_1}^-u(x),\Dt_{x_1}^+u(x)\right]\times\left[\Dt_{x_2}^-u(x),\Dt_{x_2}^+u(x)\right]\\
&= \left[\Dt_{x_1}^0u(x)-\frac{\ell_1}{2},\Dt_{x_1}^0u(x)+\frac{\ell_1}{2}\right]\times\left[\Dt_{x_2}^0u(x)-\frac{\ell_2}{2},\Dt_{x_2}^0u(x)+\frac{\ell_2}{2}\right].
\end{align*}
Here, the side lengths are given by
\[ \ell_j = \Dt_{x_j}^+u(x)-\Dt_{x_j}^-u(x) = h\Dt_{x_jx_j}u(x). \]
We recall that this is necessarily non-negative since $u$ is a convex function.  Thus we can also bound this side length by the discrete Laplacian of $u$.
\[ 0 \leq \ell_j \leq \ell_1+\ell_2 = h\Delta^hu(x). \]

Now consider any $p\in\partial u(x)$.  Then $p\in R$ as well so that
\[ \abs{p-(\Dt_{x_1}^0u(x),\Dt_{x_2}^0u(x))} \leq \frac{1}{2}\abs{(\ell_1,\ell_2)} \leq \frac{h}{\sqrt{2}}\Delta^hu(x).\]
Since $g$ is Lipschitz continuous, we can control its variation over $\partial u(x)$ by
\[ g(p) \leq g(\Dt_{x_1}^0u(x),\Dt_{x_2}^0u(x)) + \frac{h}{\sqrt{2}}L_g\Delta^hu(x) \]
As $g$ is non-negative, we also have
\[ g(p) \leq \max\{g(\Dt_{x_1}^0u(x),\Dt_{x_2}^0u(x)) + hL_g\Delta^hu(x),0\} = g^h(\nabla u)(x). \]

Now we recall that the subgradient measure $\partial u$ has all its mass concentrated on node points $x\in\G^h$.  That is,
\[ \abs{\partial u(R_x^h)} = \abs{\partial u(x)}. \]
Then we can bound the approximation of the target measure by
\bq\label{eq:gOver}
\int_{\partial u(R_x^h)}g(y)\,dy = \int_{\partial u(x)}g(y)\,dy \leq \abs{\partial u(x)} g^h(\nabla u)(x).
\eq

Finally, we need to control the approximation of the subgradient measure.  By~\cite[Lemma~4.3]{BenamouDuval_MABVP2}, the approximation of the \MA operator is an overestimation on piecewise linear grid functions so that
\bq\label{eq:MAOver}
{\det}^h(D^2u)(x) \geq \frac{1}{h^2}\abs{\partial u(x)}.
\eq

Combining~\eqref{eq:fUnder}, \eqref{eq:gOver}, and \eqref{eq:MAOver}, we obtain
\begin{align*}
F_1^h(x,u(x),u(x)-u(\cdot)) &= -g^h(\nabla u)(x) {\det}^h(D^2u)(x) + f^h(x)\\
  &\leq \frac{-\int_{\partial u(R_x^h)}g(y)\,dy + \int_{R_x^h}f(z)\,dz}{h^2}. 
\end{align*}
\end{proof}

\subsection{Convexity constraint}  To discretise the smallest eigenvalue of the Hessian, we use the generalise finite difference schemes described in~\cite{FroeseMeshfreeEigs}.  We begin by seeking an approximation of a general second directional derivative in the direction $\nu$.  We propose the form
\bq\label{eq:disc2nd}
\Dt_{\nu\nu} u(x_0) = \sum\limits_{j=1}^4a_j(u(x_j)-u(x_0)).
\eq
Here the $x_i\in\G^h$ are four points within a distance $\sqrt{h}$ of $x_0$ that align as well as possible with the direction $\nu$.  To do this, we consider orthogonal coordinate axes defined by the lines $x_0+t\nu$ and $x_0 + t\nu^\perp$.  Then the four neighbours are given by
\[ x_j = \argmax\limits_{x\in\G^h}\left\{\abs{\frac{x-x_0}{\abs{x-x_0}}\cdot\nu} \mid 0 < \abs{x-x_0} < \sqrt{h}, \, x \text{ is in the $j$th quadrant.}\right\} \]

The non-negative coefficients are chosen via Taylor expansion to ensure consistency of the expression.  The requirements for the scheme are
\bq\label{eq:coeffs}
\begin{cases}
\sum\limits_{j=1}^4 a_j(x_j-x_0)\cdot\hat{\nu} = 0 \\
\sum\limits_{j=1}^4 a_j(x_j-x_0)\cdot\hat{\nu}^\perp = 0\\
\frac{1}{2}\sum\limits_{j=1}^4 \abs{(x_j-x_0)\cdot\hat{\nu}}^2 = 1\\
a_j \geq 0.
\end{cases}
\eq
A consistent, monotone approximation is guaranteed to exist (and can be obtained explicitly) by~\cite[Theorem~13]{FroeseMeshfreeEigs}.

The approximation of $\lambda_1(D^2u(x))$ is based on the Rayleigh-Ritz characterisation of the eigenvalues of the Hessian:
\bq\label{eq:RR}
\lambda_1(D^2u) = \min\limits_{\abs{\nu}=1}\frac{\partial^2u}{\partial\nu^2}.
\eq
We consider the following discretisation of all unit vectors,
\[ \nu_j = (\cos(j d\theta),\sin(jd\theta)), \, j = 0, \ldots, N_\theta \]
where $d\theta=\bO(\sqrt{h})$ and $(N_\theta+1)d\theta = \pi$.
Then a consistent, monotone approximation of this operator is
\bq\label{eq:discLambda1}
F_2^h(x,u(x),u(x)-u(\cdot)) = -\min\limits_{j=1, \ldots, N_\theta} \Dt_{\nu_j\nu_j} u(x).
\eq

\begin{lemma}[Underestimation of convexity constraint]\label{lem:underConvex}
Let $u$ be any convex function.  Then
\[ F_2^h(x,u(x),u(x)-u(\cdot)) \leq 0, \quad x \in\G^h\cap X. \]
\end{lemma}

\begin{proof}
We will show that if $u$ is convex,
\[ \Dt_{\nu\nu}u(x_0) \geq 0 \]
for any choice of $\abs{\nu}=1$ and $x_0\in\G^h\cap X$.

Let $A = \sum\limits_{j=1}^4 a_j$ and recall that all $a_j \geq 0$ because of monotonicity.  Since $u$ is convex, we can compute
\begin{align*}
\Dt_{\nu\nu}u(x_0) &= A\left(\sum\limits_{j=1}^4 \frac{a_j}{A}\ u(x_j)-u(x_0)\right)\\
  &\geq A\left(u\left(\sum\limits_{j=1}^4 \frac{a_j}{A} x_j\right)-u(x_0)\right).
\end{align*}

The consistency conditions~\eqref{eq:coeffs} guarantee that
\[ \frac{1}{A}\sum\limits_{j=1}^4a_jx_j = \frac{1}{A}\sum\limits_{j=1}^4 a_jx_0 = x_0 \]
and therefore
\[ \Dt_{\nu\nu}u(x_0) \geq 0. \]

Then we trivially have that
\[ F_2^h(x,u(x),u(x)-u(\cdot)) = -\min\limits_{j=1, \ldots, N_\theta} \Dt_{\nu_j\nu_j} u(x) \leq 0. \]
\end{proof}

\subsection{Transport constraint}
To discretise the Hamilton-Jacobi operator that enforces the transport constraint, we will use a modified version of the Lax-Friedrichs scheme.

We begin by defining two constants related to the convex target set $Y$.  The first is an inner diameter,
\bq\label{eq:innerDiam}
D = \sup\left\{\ell \mid \text{ any line segment of length $\ell$ that intersects $\bar{Y}$ has an endpoint in $\bar{Y}$}\right\}.
\eq
Using this, we can define a discretisation parameter $M$ (independent of $h$) by
\bq\label{eq:defineM}
M = \ceil{2{\max\limits_{y\in\partial Y}\abs{y}} / {D}}.
\eq

Next, we approximations of the gradient that rely on weighted averages of the forward and backward differencing operators.
\bq\label{eq:gradapprox}
\nabla^h_{ij}u(x) = \left(\left(1-\frac{i}{M}\right)\Dt_x^-u(x)+\frac{i}{M}\Dt_x^+u(x), \left(1-\frac{j}{M}\right)\Dt_y^-u(x)+\frac{j}{M}\Dt_y^+u(x)\right).
\eq

Finally, we can define the following modified Lax-Friedrichs scheme, which is consistent and monotone.
\bq\label{eq:HJapprox}
F_3(x,u(x),u(x)-u(\cdot)) = \min\limits_{0 \leq i,j \leq M}H(\nabla^h_{ij}u(x)) - hL_H\Delta^hu(x).
\eq

\begin{lemma}[Underestimation of transport constraint]\label{lem:underTransport}
Let $u$ be any convex function with $\partial u(\G^h)\cap\bar{Y} \neq \phi$.  Then for any $x\in\G^h\cap X$,
\[ F_3(x,u(x),u(x)-u(\cdot)) \leq 0. \]
\end{lemma}
\begin{proof}
We begin by defining the following cone,
\[ v(x) = \sup\left\{v \text{ convex}\mid v(x)=u(x),\, v(x\pm h e_1)=u(x\pm he_1), \, v(x\pm he_2) = u(x \pm he_2) \right\}. \]
Notice that $v \geq u$ and thus $\partial u(x) \subset \partial v(x)$ where $\partial v(x)$ is given explicitly by the following rectangle.
\begin{align*} R \equiv \partial v(x) &= \left[\Dt_x^-v(x),\Dt_x^+v(x)\right]\times\left[\Dt_y^-v(x),\Dt_y^+v(x)\right]\\
&=\left[\Dt_x^-u(x),\Dt_x^+u(x)\right]\times\left[\Dt_y^-u(x),\Dt_y^+u(x)\right].
\end{align*}
In particular, it must be the case that $R\cap \bar{Y}$ is non-empty since $\partial u(x) \subset R$.

Moreover, we can find \emph{a priori} bounds on the size of this rectangle since $u$ is convex and we have information about its subgradient.  In particular, for $j = 1,2$ we have
\[ \abs{u(x\pm he_j)-u(x)} \leq h \max\limits_{y\in\partial Y} \abs{y} \]
and thus the side lengths of the rectangle are bounded by
\[ \Dt_{x_j}^+u(x)-\Dt_{x_j}^-u(x) \leq \abs{\Dt_{x_j}^+u(x)}+\abs{\Dt_{x_j}^-u(x)} \leq 2 \max\limits_{y\in\partial Y}\abs{y}. \]

Now we notice that $\nabla^h_{ij}u(x)$, $i,j=0,\ldots,M$ is a discretisation of the rectangle $R$ with stepsize
\[ \frac{\Dt_x^+u(x) - \Dt_x^-u(x)}{M} \leq \frac{2 \max\limits_{y\in\partial Y}\abs{y}}{\ceil{2 \max\limits_{y\in\partial Y}\abs{y}/D}} \leq D. \]
Recalling the definition of $D$ as the inner diameter of $Y$, we find that
\[ \nabla^h_{ij}u(x) \in \bar{Y} \]
for some choice of $i^*,j^*=0,\ldots,M$.

Because $H$ is a defining function for $Y$, we are guaranteed that
\[ H(\nabla^h_{i^*j^*}u(x)) \leq 0 \] 
and therefore
\[ \min\limits_{i,j=0,\ldots,M}H(\nabla^h_{ij}u(x)) \leq 0. \]

Finally, we observe that the discrete Laplacian satisfies
\[ \Delta^hu(x) = \Dt_{x_1x_1}u(x)+\Dt_{x_2x_2}u(x), \]
which is non-negative since $u$ is convex.

Combining these results, we obtain
\begin{align*}
F_3^h(x,u(x),u(x)-u(\cdot)) &= \min\limits_{i,j=0,\ldots,M}H(\nabla^h_{ij}u(x)) - hL_H\Delta^hu(x) \leq 0. 
\end{align*}
\end{proof}

\section{Conclusions}\label{sec:conclusions}
In this article, we introduced and analysed a new notion of weak solution for the second boundary value problem for the \MA equation.  This definition relied on the usual concept of a viscosity subsolution applied to a modified PDE that simultaneously enforced the \MA equation, convexity constraint, and global constraint on the solution gradient.  We proved that these viscosity subsolutions are equivalent to solutions of the original optimal transportation problem.

Using this new formulation, we showed that the Barles-Souganidis convergence framework, which requires PDEs to satisfy a comparison principle, can be modified to apply to viscosity subsolutions of the second boundary value problem for \MA (which does not satisfy a comparison principle). In particular, we demonstrated that consistent, monotone, underestimating schemes will converge to the weak solution of the optimal transportation problem.  Many existing numerical methods, that were previously not equipped with convergence proofs, can be modified to fit into this framework.  We also provided a concrete example showing how existing discretisations can be easily modified to fit this convergence framework.

\bibliographystyle{plain}
\bibliography{OTBC_Viscosity}
\end{document}